\documentclass{amsart}
\usepackage[utf8]{inputenc}
\usepackage{amsmath, amsfonts, amsthm, amssymb}
\usepackage{xcolor}
\usepackage{pgfplotstable}
\usepackage{pgfplots}
\usepackage{float}
\usepackage{subfloat}
\usepackage{caption}
\usepackage{subcaption}
\usetikzlibrary{patterns}
\newtheorem{theorem}{Theorem}[section]
\newtheorem{proposition}[theorem]{Proposition}
\newtheorem{lemma}[theorem]{Lemma}

\newtheorem{conjecture}[theorem]{Conjecture}

\theoremstyle{definition}

\newtheorem{remark}[theorem]{Remark}

\newcommand{\R}{{\mathbb{R}}}
\newcommand{\fp}{\text{FP}}
\newcommand{\pfp}{\text{PFP}}
\newcommand{\Sb}{\mathbb{S}}
\newcommand{\Fc}{{\mathcal{F}}}
\newcommand{\Pc}{{\mathcal{P}}}
\newcommand{\sign}{\text{sign}}
\newcommand{\Xperp}[1]{{X_{#1}^{\perp}}}
\newcommand{\eps}{\epsilon}
\newcommand{\pit}{\frac{\pi}{2}}

\title[Phase transitions for frame potentials]{Phase transitions for the minimizers of the $p^{th}$ frame potentials in $\mathbb{R}^2$}

\date{\today}

\begin{document}

\author{Radel Ben Av }
\address{Department Of Computer Science, Holon Institute Of Technology, Holon, Israel, 5810201}\email{benavr@hit.ac.il}

\author{Xuemei Chen }
\address{Department  of  Mathematics  and  Statistics,  University  of  North  Carolina  Wilmington,  Wilmington,  NC  28409, USA}\email{chenxuemei@uncw.edu}

\author{Assaf Goldberger}
\address{School of Mathematical Sciences, Tel-Aviv University, Tel-Aviv, Isreal, 69978} \email{assafg@tauex.tau.ac.il}

\author{Shujie Kang}
\address{Department of Mathematics,
University of Texas at Arlington, Arlington TX 76019, USA}
\email{shujie.kang@uta.edu}

\author{Kasso A.~Okoudjou}\thanks{K.A.O. was partially supported by a grant from   the National Science Foundation grant DMS 1814253.}
\address{Department of Mathematics,
Tufts University, Medford MA 02131, USA} \email{Kasso.Okoudjou@tufts.edu}

\begin{abstract}
Given $N$  points $X=\{x_k\}_{k=1}^N$ on the unit circle in $\mathbb{R}^2$ and a number $0\leq p \leq \infty$ we investigate the minimizers of the functional $\sum_{k, \ell =1}^N |\langle x_k, x_\ell\rangle|^p$. While it is known that each of these minimizers is a spanning set for $\mathbb{R}^2$, less is known about their number as a function of $p$ and $N$ especially for relatively small $p$. In this  paper we show that there is unique minimum for this functional for all $p\leq \log 3/\log 2$ and all odd $N\geq 3$. In addition, we present some numerical results suggesting the emergence of a phase transition phenomenon for these minimizers. More specifically, for $N\geq 3$ odd, there exists a sequence of number of points $\log 3/\log 2=p_1< p_2< \hdots < p_N\leq 2$ so that a unique (up to some isometries)  minimizer exists on each sub-intervals $(p_k, p_{k+1})$. %In addition we conjecture that $\lim_{k\to \infty}p_{2k+1}=2$.

\end{abstract}


\maketitle \pagestyle{myheadings} \thispagestyle{plain}
\markboth{R. BEN AV, X. CHEN, A. GOLGBERGER, S. KANG, And K. A. OKOUDJOU}{Phase transition for the minimizers of the $p^{th}$frame potentials in $\mathbb{R}^2$}

\maketitle

\section{Introduction}
% \textcolor{red}{to do: background introduction}

% \subsection{Notation}
Let $S(N,d)$ be the collection of all ordered multisets of   $N$ unit-norm vectors in $\R^d$.  For any $p\in (0, \infty]$, the \emph{$p$-frame potential} of $X=\{x_k\}_{k=1}^{N}\in S(N,d)$ is defined as 
\begin{equation}\label{eq:pfrpot}
\fp_{p,N,d}(X):=\left\{\begin{array}{ll}\displaystyle\sum_{k=1}^{N}\sum_{\ell\neq k}^N|\langle x_k,x_{\ell}\rangle|^p,\, &\text{when}\, p<\infty\\
\displaystyle\max_{ k\neq \ell} |\langle x_{k},x_{\ell}\rangle|,&\text{when }p=\infty.\\
\end{array}\right.
\end{equation} 

The continuity of $\fp_{p, N, d}$ and the compactness of the unit sphere guarantee the existence of a solution to the following optimization problem: 

%Finding the infimum of the $p$-frame potential among all vectors in $S(N,d)$ for fixed $p, N, d$ as well as its . Since the unit sphere is compact and the function $FP$ is a continuous function, the infimum can be achieved. So we would like to find 
\begin{equation}\label{equ:min}
\Fc_{p,N,d}:=\min\limits_{X\in S(N,d)}\fp_{p,N,d}(X).
\end{equation}
Finding this minimum value and the corresponding minimizers has been  the subject of several recent investigations \cite{AvGoDuSt, BGMPO22, BilMat18, Bukcox18, CGGKO, CHS21, Glaz19}.
%We note that if $X \in S(N,d)$ whose $p$-frame potential achieves this minimum. 
Observe that if $X=\{x_i\}_{i=1}^N $, henceforth referred to as an $N$ point configuration, is a minimizer of $\fp_{p, N, d}$, then so is $X'= \{s_1Ux_{\pi_1},\cdots,s_NUx_{\pi_N}\}$  where $U$ is any orthogonal matrix, $\pi$ is a permutation on $\{1,2,\cdots,N\}$, and $s_i\in \{-1,1\}$. One can check that this correspondence between $X$ and $X'$ defines an equivalence relation on the set of $N$ points configurations.  As such, we say the solution to ~\eqref{equ:min} is unique if it is unique up to this equivalence relation.
%given by: $X=\{x_i\}_{i=1}^N \sim X'=\{x'_i\}_{i=1}^N$ if and only if 
%We define an equivalence relation 
%$X'=\{x'_i\}_{i=1}^N= \{s_1Ux_{\pi1},\cdots,s_1Ux_{\pi N}\}$  where $U$, $\pi$ and $s_i$ are chosen as above. 
%is any orthogonal matrix, $\pi$ is a permutation and $s_i\in \{-1,1\}$. It is easy to see that the $p$-frame potentials of $X,X'$ are equal if they are equivalent. 
%We say the minimizer is unique if it is unique up to the equivalence relation.

In \cite{CGGKO}, the authors proved that for a given dimension $d$ and natural number $N$ the minimizers of \eqref{equ:min} are universal in the sense that the same configurations remain  minimizers (not necessarily unique)  for a variety of kernel functions including the $p$ potential for large ranges of $p$.
 This can be viewed as a special case of a result proved for a broader class of potentials in \cite{CohKum2007}. This universality property of the minimizers of \eqref{equ:min} can also  be viewed as a phase transition phenomenon. For example, with $d=2$ and $N=3$, it is proven in \cite[Corollary 3.7]{EhlOko2012} that for $0<p<\frac{\log(3)}{\log(2)}$, the unique minimizer of $\fp_{p,3,2}$ is $\{e_1,e_1,e_2\}$, where $e_1=(1,0)$ and $e_2=(0,1)$; for $p>\frac{\log(3)}{\log(2)}$, the unique minimizer of $\fp_{p,3,2}$ is $\{e_1, (\cos\frac{\pi}{3},\sin\frac{\pi}{3}), (cos\frac{2\pi}{3},\sin\frac{2\pi}{3})\}$. 
 
 These two types of optimal configurations will be used throughout the paper. For convenience, we let
$$\Xperp{N}:=\left\{\begin{array}{ll}\{\underbrace{e_1,e_1,\cdots,e_1}_{N/2},\underbrace{e_2,e_2,\cdots,e_2}_{N/2}\},&N \text{ is even}\\
\{\underbrace{e_1,e_1,\cdots,e_1}_{(N+1)/2},\underbrace{e_2,\cdots,e_2}_{(N-1)/2}\},&N \text{ is odd}
\end{array}\right.$$
 and 
 $$X_N^{(h)}:=\{(\cos\frac{j\pi}{N},\sin\frac{j\pi}{N}),j=0,1,2,\cdots,N-1\}.$$
 
 When in $\R^2$, each unit norm vector is identified by its angle. We will conveniently adopt the notation that $x(\theta)=\exp(i\theta)=(\cos\theta,\sin\theta)$, and for a set $\Theta$
 \begin{equation}
 X(\Theta)=\{x(\theta):\theta\in\Theta\}.
 \end{equation}
 With this notation, for example, $\Xperp{5}=X(0,0,0,\pit,\pit).$ Moreover, for $0<p<\infty$,
 \begin{equation}\label{equ:fp2}
 \fp_{p,N,2}(X(\theta_1,\theta_2,\cdots,\theta_N))=\sum_{k=1}^{N}\sum_{\ell\neq k}^N| \cos(\theta_k-\theta_{\ell})|^p.
 \end{equation}
 
The phase transitioning phenomenon is different depending on whether $N$ is even or odd. It is proven in \cite{CGGKO} that for $N=4$, $\Xperp{4}$ is the minimizer for $0<p<2$ and $X_4^{(h)}$ minimizes $\fp_{p,4,2}$ for $p>2$.
More generally,  \cite[Theorem 4.9]{EhlOko2012} implies that for even $N$, $\Xperp{N}$ is the unique minimizer if $0<p<2$.
Our understanding of the minimizers for achieving $\Fc_{p, N, 2}$ is still largely incomplete for odd $N\geq 5$  when $p$ is relatively small. %A more recent result in \cite{GP20} shows for any odd $N$, the unique minimizer of \eqref{equ:min} is $\Xperp{N}$ for $0<p<1.3$. However, it is believed that the result is true for $p$ up to a number much closer to 2. 

Our goal in this paper is to focus on the two dimensional case and investigate this phase transition behavior when $N\geq 5$ for small values of $p$. For the sake of completeness, we summarize the known results  in Table~\ref{tab:R2known}. 

%Then, the solutions to ~\eqref{equ:min} when $d=2$ are:
\begin{table}[hbt]
\caption{\label{tab:R2known}Existing results for $\R^2$.}
\begin{tabular}{|c|c|c|}
\hline
$p$ & $N\geq5$ & minimizer of $\fp_{p,N,2}$ \\ \hline
$0<p<2$ & even & $X_N^{\perp}$ \cite[Theorem 4.9]{EhlOko2012}\\ \hline
$0<p\leq1.3$ & odd & $X_N^{\perp}$, unique \cite[Theorem 4.2]{GP20} \\ \hline
%2 & $N\ge 2$ & tight frames \\ \hline
$2<p\leq4\lfloor N/2\rfloor-2$, even &any $N$  & $X_N^{(h)}$  \cite[Theorem 3.7]{CGGKO}\\ \hline
$p>4\lfloor N/2\rfloor-2$ & any $N$ & $X_N^{(h)}$, unique \cite[Theorem 3.7]{CGGKO} \\ \hline
\end{tabular}
\end{table}

Based on Table~\ref{tab:R2known},  to completely solve the optimization problem~\eqref{equ:min} when $d=2$ one needs to consider the following two cases: 
%One of the goals  of this paper is to shed some new lights on the following  still  unanswered cases 
\begin{enumerate}
    \item $N\geq 5$ and odd with $p \in (1.3, 2)$,
    \item $N\geq 5$ and $p\in (2, 4\lfloor N/2\rfloor -2]$ and $p$ not an even integer.
\end{enumerate} 
One of the objectives of this paper to shed some new lights on these cases. 
In particular, we will extend a result proved in  \cite{GP20}  which states that if $0<p\leq1.3$ the absolute minimum of $\fp_{p,N,2}$ for odd $N$ is $(N-1)^2/2$, which is achieved by $X_N^\perp$.  We show that the result still holds when  $p\le \frac{\log 3}{\log 2}\approx 0.58$.\\

\begin{theorem}\label{thm:main}
	Suppose $0< p\le \frac{\log 3}{\log 2}$ and $N\geq3$ is odd, then the absolute minimum of $\fp_{p,N,2}$ is $(N-1)^2/2$, and the unique optimal configuration is $X_N^\perp$.
\end{theorem}

According to  Theorem~\ref{thm:main} and Table~\ref{tab:R2known}, $p\in \left(\frac{\log 3}{\log 2}, 6 \right)\setminus \{2, 4\}$ is the remaining case  to have a solution to ~\eqref{equ:min} when $N=5$. We will present some numerical results dealing with this case, leading to a conjecture (Conjecture~\ref{conj:1}) on the solution to the problem in this case. As a result, we will see that the minimizers of ~\eqref{equ:min} give rise to a phase transition phenomenon. More generally,  our numerical results show that this phase transition phenomenon persists and results in an intricate behavior of the minimizers of $\fp_{p, N,2}$ for $p\in (0, 2)$ when $N\geq 5$ and odd.

The rest of the paper is organized as follows. In Section~\ref{sec2} we state and prove some technical results leading to the proof of Theorem \ref{thm:main}. In  the first part of  Section~\ref{sec3} we show that for $0<p<2$, as $N\to \infty$ the minimizers appear to be $X_N^{\perp}$. When $N$ is even, this result was already established in \cite{EhlOko2012}, so our numerical results indicate that this is still true when $N$ is odd. In either case, the minimizers are asymptotically close to being tight frames \cite{BenFic2003}.
%copies of  the orthonormal bases {\color{red} What does this statement exactly mean? Are the optimal configurations close in some sense to $X_N^\perp$?}
Finally, in second part of Section~\ref{sec3} we present the numerical results pertaining to the general behaviour of the solution of ~\eqref{equ:min} when $0<p< 4\lfloor N/2\rfloor -2$.

\section{The optimal configuration for  $p\le \log 3/\log 2$ with odd $N$}\label{sec2}

As summarized in Table \ref{tab:R2known}, when $0<p\leq2$ and $N$ is even, the  absolute minimum of $\text{FP}_{p,N,2}$  is $(N^2-2N)/2$, uniquely minimized by $\Xperp{N}$~\cite[Theorem 4.9]{EhlOko2012}.
 In this section, we use an induction argument to deal with the case of odd $N$, and provide a proof of Theorem \ref{thm:main}. This induction argument is inspired by the method used in \cite{GP20}, which we refine and extend to a larger range of $p$. 

In view of \eqref{equ:fp2}, we define 
\begin{equation}\label{equ:W}
W_p(\theta):=|\cos \theta|^p,
\end{equation}
so $\fp_{p,N,2}(X(\theta_1,\theta_2,\cdots,\theta_N))=\sum_{k=1}^N\sum_{\ell\neq k}^N W_p(\theta_k-\theta_\ell).$

The potential of interest will be compared to the `linearized' potential function 
$$F_V(\theta_1,\ldots,\theta_N):=\sum_{k=1}^N\sum_{\ell\neq k}^N V(\theta_k-\theta_\ell),$$
where 
\begin{equation}\label{equ:V}
V(\theta):=\frac{2}{\pi}|\theta-\frac{\pi}{2}|
\end{equation} is defined for $x\in [0,\pi]$ and extended with period $\pi$ to the real line.

 In a small neighborhood of $\theta=0$ we have $W_p(\theta)\ge V(\theta)$, while in a small neighborhood of $\theta=\frac{\pi}{2}$ the inequality is reversed. Let $\theta_p$ be the maximum value such that $V(\theta)\le W_p(\theta)$ for all $\theta\in [0,\theta_p]$. See Figure \ref{fig:VW}. A simple convexity test shows that when $p=\frac{\log 3}{\log 2}$ we have $\theta_p=\frac{\pi}{3}=\frac{\pi}{2}-\frac{\pi}{6},$ and when  $p<\frac{\log 3}{ \log 2}$ we have $\theta_p>\frac{\pi}{3}$. 

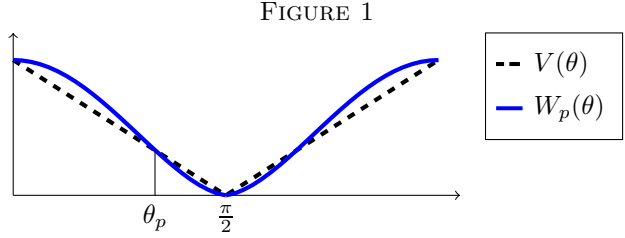
\begin{figure}[hbt]
\caption{}\label{fig:VW}
\begin{tikzpicture}[scale=1.8,
blueline/.style={shape=rectangle, inner sep=0pt, minimum height=0pt, minimum width=10pt,draw=blue, line width=1.5},
  blackdash/.style={shape=rectangle, inner sep=0pt, minimum height=0pt, minimum width=10pt, line width=1.5, draw=black,dashed}]
\draw[->](0,0)--(3.3,0);
\draw[->](0,0)--(0,1.2);
\draw (4.5,0);
\draw[ultra thick,dashed](0,1)--(pi/2,0)--(pi,1);
\draw[ultra thick, blue] plot[domain=0:1.571](\x,{(cos(\x r))^1.58});
\draw[ultra thick, blue] plot[domain=1.571:pi](\x,{(-cos(\x r))^1.58});
\matrix [draw,below left] at (current bounding box.north east) {
  \node [blackdash,label=right:$V(\theta)$] {}; \\
  \node [blueline,label=right:$W_p(\theta)$] {}; \\
};
\draw (pi/2,0) node[below]{$\frac{\pi}{2}$};
\draw (pi/3,1/3)--(pi/3,0)  node[below]{$\theta_p$};
\end{tikzpicture}
\end{figure}

We will need the following observation on $F_V$.

\begin{proposition}\label{prop:54}
For odd $N$, the configuration 
$\Theta_N^\perp=(0,\ldots,0,\frac{\pi}{2},\ldots,\frac{\pi}{2})$ with $(N+1)/2$ 0's is a
global minimizer of $F_V$. The global minimum is thus $F_V(\Theta_N^\perp)=(N-1)^2/2$.
\end{proposition}

\begin{proof}
    For any configuration $\Theta=\{\theta_1,\ldots,\theta_N\}$ on the unit circle let 
    us define a relation $\sim_\Theta$ on the set $\{1,2,\ldots,N\}$ by saying that $i\sim_\Theta j$ if and only if $\theta_i-\theta_j \equiv 0 \mod \frac{\pi}{2}.$ This is an equivalence relation. We will show first that $F_V$ has a global minimizer $\Phi=\{\phi_1,\ldots,\phi_N\}$ whose relation $\sim_\Phi$ has just one equivalence class. This in turn implies that all angles $|\phi_i-\phi_j|$ are either $0$ or $\frac{\pi}{2}$. Then after reordering and a translation we assume without loss of generality that $\Phi=\{0,\ldots,0,\frac{\pi}{2},\ldots,\frac{\pi}{2}\}$ with $r$ 0's and $N-r$ $\frac{\pi}{2}$'s. It is easy to see that $F_V(\Phi)$ is minimal when $r$ is as close as possible to $\frac{N}{2}$, from which the proposition will follow.
    
    So let's assume that we have chosen $\Phi$ to be a global minimizer such that the number $m$ of $\sim_\Phi$ equivalence classes is as minimal as possible. We need to show that $m=1$. If $m>1$, then after reordering we have a partition $A=\{1,\ldots,k\},B=\{k+1,\ldots,N\}$ such that  any $i \in A$ and any $j\in B$  are not $\sim_\Phi$-equivalent. Let $\Phi_\alpha:=\{\phi_1+\alpha,\ldots,\phi_k+\alpha,\phi_{k+1},\ldots,\phi_N\}$, then the function
    $$ H(\alpha):=F_V(\Phi_\alpha), $$
    is linear in a neighborhood of $\alpha=0$. By the minimality of $F_V(\Phi)$, $H(\alpha)$ must be constant, as long as it is linear. So $\Phi_\alpha$ is still a global minimizer, up to the point where some $\phi_i+\alpha_0\equiv \phi_j\mod \frac{\pi}{2}$, $i\le k<j$. Note that the number of equivalence classes of $\sim_{\Phi_{\alpha_0}}$ is now reduced to at most $m-1$, in contrast to the minimality of $m$. So $m=1$, and the proof is complete.
    
\end{proof}

\begin{remark}
    There is no claim here for the uniqueness of the global minimizer of $F_V$. In fact $(0,\ldots,0,\alpha,\frac{\pi}{2},\ldots,\frac{\pi}{2})$ with $\frac{N-1}{2}$ many 0's is a global minimizer for every $0\le \alpha\le \frac{\pi}{2}$.
\end{remark}

We need few technical lemmas from calculus of one variable.

\begin{lemma}\label{GG}
	Let $1\le p\le 2$, and let $\theta_c\in [0,\frac{\pi}{2}]$ be  the angle such that $\cos^2\theta_c=\frac{p-1}{p}$. Let \begin{equation}\label{equ:G}
	G(\theta)=\cos^{p-1}(\theta)\sin(\theta)
	\end{equation} be defined on $[0,\frac{\pi}{2}]$. Then 
	\begin{enumerate}
	    \item[(a)] The function $G(\theta)$ is increasing in the interval $[0,\theta_c]$, and decreasing in the interval $[\theta_c,\pit]$.
	    \item[(b)] For $0\le \theta\le \frac{\pi}{4}$ we have $G(\pit-\theta)\ge G(\theta)$ with equality only at the endpoints.
	    \item[(c)] For $1\le p\le 1.73$ and for every $\alpha>0$ such that $\theta_c+\alpha<\pit$, $G(\theta_c-\alpha)\ge G(\theta_c+\alpha)$.
	    \item[(d)] For $1\le p\le 1.73$, if $x\le \theta_c \le y$ are such that $G(x)=G(y)$, then $\frac{1}{2}(x+y)\le \theta_c$.
	\end{enumerate}
\end{lemma}

\begin{proof}
We note that $\theta_c \in [\frac{\pi}{4}, \pit].$

	Part (a) follows from the simple fact that   $G'(\theta)=p\cos^{p-2}\theta(\cos^2\theta-\frac{p-1}{p})$. We see that $G'$ vanishes only at $\theta_c$. 
	
	For the second assertion note that
	\begin{align*}
	 G\left(\pit-\theta\right)-G(\theta)&=\sin^{p-1}\theta\cos\theta-\cos^{p-1}\theta\sin\theta\\ &=\sin^{p-1}\theta\cos^{p-1}\theta (\cos^{2-p}\theta-\sin^{2-p}\theta)\ge 0.
	\end{align*}
	
	To prove (c) we compute few higher order derivatives of $G$:
	\begin{align*}
	    G''(\theta) &= -p^2\sin\theta\left( \cos^{p-1}\theta+\frac{(p-1)(2-p)}{p^2}\cos^{p-3}\theta  \right)\le 0,\\
	    G'''(\theta) &= \cos^{p-4}\theta (-p^3\cos^4\theta +(p-1)(p^2+(2-p)^2)\cos^2\theta-(p-1)(2-p)(3-p)).
	\end{align*}
	We now show that $G'''(\theta)\le 0$. It suffices to show that the quadratic polynomial $P_p(u):=-p^3u^2+(p-1)(p^2+(2-p)^2)u-(p-1)(2-p)(3-p)<0$ for all $u\in[0,1]$. This polynomial is maximized at $u_{max}=(p-1)(p^2+(2-p)^2)/2p^3\in[0,1]$. Substituting $P_p(u_{max})$ we obtain a rational function in $p$: $6p + 28/p - 16/p^2 + 4/p^3 - 22$ which is easily seen to be negative for $1\le p\le 1.73$.
	
	Now it follows that $G'(\theta)$ is concave. Let 
	$$ SG(\alpha)=G(\theta_c-\alpha)-G(\theta_c+\alpha)$$ 
	be defined for $0\le\alpha\le \pit-\theta_c$. Then
	$$SG'(\alpha)=-G'(\theta_c-\alpha)-G'(\theta_c+\alpha)\ge -2G'(\theta_c)=0$$ by the concavity of $G'$. It follows that $SG(\alpha)\ge SG(0)=0$ from which assertion (c) follows.
	
	To prove (d), suppose that $G(x)=G(y)$ for $x\le\theta_c\le y$. Letting $\alpha=y-\theta_c\in[0,\frac{\pi}{2}-\theta_c]$, part (c) implies that $G(2\theta_c-y)=G(\theta_c-\alpha)\ge G(\theta_c+\alpha)= G(y)=G(x)$. Since both $x$ and $2\theta_c-y$ belong to $[0,\theta_c]$, by the monotonicity in part (a) we must have $x\le 2\theta_c-y$, thus arrives at our conclusion.
	\end{proof}

\begin{lemma}\label{KK}
	For $1\le p\le 2$ and $0\le \nu\le \frac{\pi}{2}$
	the function $K(\theta)=\cos^p(\theta)+\cos^p(\frac{\pi}{2}-\nu-\theta)$ in the interval $0\le \theta\le \frac{\pi}{2}-\nu$ is minimized (only) at the endpoints. In particular $K(\theta)\ge 1+\sin^p\nu$.
\end{lemma}

\begin{proof}
	We notice
	$$K'(\theta)=-pG(\theta)+pG\left(\pit-\nu-\theta\right),$$
where $G$ is defined in \eqref{equ:G}. Moreover, $K'(0)=-pG(0)+pG\left(\pit-\nu\right)=pG\left(\pit-\nu\right)\geq0$, and $K'(\pit-\nu)=-pG\left(\pit-\nu\right)\le 0$.
In particular the endpoints are local minima. We will show that $K'(\theta)$ vanishes internally only at the midpoint $\theta=\frac{1}{2}(\pit-\nu)$, from which the lemma follows. 

Suppose that $0<\theta_0<\pit-\nu$ is an angle such that $K'(\theta_0)=0$. Then by Lemma \ref{GG}(a) either that $\theta_0=\pit-\nu-\theta_0$, or that $\theta_c$ lies between $\theta_0$ and $\pit-\nu-\theta_0$. In the first case $\theta_0$ is the midpoint. In the second case, by symmetry we may assume that $0<\theta_0\le \theta_c\le \pit-\nu-\theta_0$.
Since $0<\theta_0\le \frac{\pi}{4}$, Lemma \ref{GG}(b) implies 
\begin{equation}\label{equ:c1}
G(\theta_0)\le G(\pit-\theta_0),
\end{equation} with equality only at $\theta_0=\frac{\pi}{4}$. On the other hand, we have $\theta_c\le \pit-\nu-\theta_0\le \pit-\theta_0$, so Lemma \ref{GG}(a) implies 
\begin{equation}\label{equ:c2}
G(\pit-\nu-\theta_0)\geq G(\pit-\theta_0).
\end{equation}
\eqref{equ:c1} and \eqref{equ:c2} contradicts to $K'(\theta_0)=0$, unless $\theta_0=\frac{\pi}{4}$ and $\nu=0$. But in this case again $\theta_0$ is the midpoint. This completes the proof.
\end{proof}

\begin{lemma}\label{LL}
	For $\frac{4}{3}\le p\le \frac{\log 3}{\log 2}$, the function $L(\alpha)=1+2\sin^p\left(\frac{\alpha}{2}\right)+\cos^p\alpha$, $\alpha\in [0,\frac{\pi}{3}]$ is minimized at endpoints. In particular $L(\alpha)\ge 2$.
\end{lemma}

\begin{proof}
	We make a change of variables, $u=\sin^2\left(\frac{\alpha}{2}\right)$, and $\cos\alpha$ becomes $1-2u$. Thus $$L(\alpha)=M(u):=1+2u^{p/2}+(1-2u)^p, \ \ u\in\left[0,\frac{1}{4}\right].$$
	We compute
	\begin{align*}
		M'(u)&= pu^{p/2-1}-2p(1-2u)^{p-1},\\
		M''(u)&=p\left(\frac{p}{2}-1\right)u^{p/2-2}+4p(p-1)(1-2u)^{p-2},
	\end{align*} 
	and notice that $M''$ is an increasing function in the interval $[0,\frac{1}{4}]$. Moreover $M''(0^+)=-\infty$ and $M''(\frac{1}{4})=\left(\frac{1}{2}\right)^{p-3}p(3p-4)>0$. Thus $M''$ has a unique zero at some point $0<u_1<\frac{1}{4}$. Consequently $M'$ has a unique local minimum at $u_1$, and it is decreasing in $[0,u_1]$ and increasing at $[u_1,\frac{1}{4}]$. Since $M'$ is positive around $u=0$ and $M'\left(\frac{1}{4}\right)=0$, it is positive at $(0,u_0)$ and negative at $(u_0,\frac{1}{4})$ for some $0<u_0<u_1$. It follows that $M$ has a unique local maximum at $u=u_0$, and that this function is minimized at the endpoints. 
	
	$L(0)=2$ and $L(\pi/3)=1+3(\frac{1}{2})^p\geq1+3(\frac{1}{2})^{\frac{\log3}{\log2}}=2$.
\end{proof}

\begin{remark}
    As the proof shows, this function is minimized at the endpoints for $\frac{4}{3}\le p\le 2$, but
	crucially for $p\le\frac{\log 3}{\log 2}$ the values at the endpoints are $\ge 2$. This is no longer true for larger $p$.
\end{remark}

\begin{lemma}\label{PP}
	For $\frac{4}{3}< p\le 1.73$, the function 
	$P(\alpha)=\cos^p(\alpha)+\cos^p(\frac{2\pi}{3}-\alpha)$ for $\frac{\pi}{3}\le \alpha\le \pit$ is minimized at $\alpha=\frac{\pi}{3}$. In particular $P(\alpha)\ge \frac{2}{3}$.
\end{lemma}

\begin{proof}
	We change variables to $\alpha=\frac{\pi}{3}+\beta$, so we need to minimize $Q(\beta)=\cos^p\left( \frac{\pi}{3}+\beta\right)+
	\cos^p\left( \frac{\pi}{3}-\beta\right)$, $0\le \beta\le \frac{\pi}{6}$. Differentiating,
	$$Q'(\beta)=-pG\left(\frac{\pi}{3}+\beta\right)+pG\left(\frac{\pi}{3}-\beta\right).$$ If the minimum of $Q$ occurs internally at some $\beta_0$, we must have by Lemma \ref{GG}(a) that 
	$$ \frac{\pi}{3}-\beta_0\le \theta_c\le \frac{\pi}{3}+\beta_0.$$
	Note that for $p$ in the range of the lemma, $\frac{\pi}{4}\le \theta_c<\frac{\pi}{3}$. Since the average of $x=\frac{\pi}{3}-\beta_0$ and $y=\frac{\pi}{3}+\beta_0$ is $\frac{\pi}{3}>\theta_c$, we obtain a contradiction to Lemma \ref{GG}(d), which proves that the minimum occurs at the end points.
	
Due to the restriction on $p$, $P(\pi/3)=(1/2)^{p-1}<(\frac{\sqrt{3}}{2})^p=P(\pi/2)$. So $P(\alpha)\geq P(\pi/3)\geq 2/3$.
\end{proof}

\begin{lemma}\label{RR} Suppose  that $\frac{4}{3}< p\le 1.73$ and  $0\le \alpha\le \frac{\pi}{3}$. The function  $R(\rho)=\sin^p\rho+\sin^p(\alpha-\rho)$ defined on  $0\le \rho\le \frac{\alpha}{2}$ is minimized at one of the endpoints. 
\end{lemma}

\begin{proof}
	We differentiate 
	$$R'(\rho)=pG\left(\pit-\rho\right)-pG\left(\pit-\alpha+\rho\right).$$
	If $R$ achieves its minimum value at an interior point $\rho_0$ then $G\left(\pit-\rho_0\right)=G\left(\pit-\alpha+\rho_0\right)$. By Lemma \ref{GG} there are two cases:
	\begin{enumerate}
	    \item[(i)] $\pit-\rho_0=\pit-\alpha+\rho_0$, which is $\rho_0=\frac{\alpha}{2}$, and this is not an internal point. Or,  
	    \item[(ii)]
	$ \pit-\rho_0\ge \theta_c \ge \pit-\alpha+\rho_0.$ 
	\end{enumerate}
	In the second case the average is $\pit-\frac{\alpha}{2}\ge \frac{\pi}{3}$. Again since $\theta_c<\frac{\pi}{3}$, by Lemma \ref{GG}(d) this is a contradiction. So there is no interior point achieving the minimum value, and the lemma is proved.
\end{proof}

A corollary of this is the following:

\begin{lemma}\label{lem:eps}
	Suppose that $\eps_i$, $i=1,2,3$ are acute angles who sum up to $\pit$. Assume that $1\le p\le \frac{\log 3}{\log 2}$. Then
	$$\sum_{i=1}^3\sin^p\eps_i \ge 1.$$
\end{lemma}

\begin{proof}
	It is enough to assume $p=\frac{\log 3}{\log 2}$. We may reorder so that $\epsilon_1\le \epsilon_2\le \epsilon_3$. Then 
	$\epsilon_1\le (\epsilon_1+\epsilon_2)/2\le \frac{\pi}{6}$ and invoking Lemma \ref{RR} with $\rho=\epsilon_1$ and $\alpha=\epsilon_1+\epsilon_2$, we see that $$\sin^p\epsilon_1+\sin^p\epsilon_2\ge \min\{2\sin^p\left(\frac{\epsilon_1+\epsilon_2}{2}\right),\sin^p(\epsilon_1+\epsilon_2)\}.$$ 
	If $\sin^p(\epsilon_1+\epsilon_2)\le 2\sin^p\left(\frac{\epsilon_1+\epsilon_2}{2}\right)$ then
	$$\sum_i \sin^p\eps_i\ge \sin^p(\epsilon_1+\epsilon_2)+\sin^p\epsilon_3\ge \sin^2(\epsilon_1+\epsilon_2)+\sin^2\epsilon_3=1,$$ and we are done. Otherwise
	$$\sum_i \sin^p\eps_i\ge 2\sin^p\left(\frac{\epsilon_1+\epsilon_2}{2}\right)+\sin^p\eps_3=2\sin^p\left(\frac{\epsilon_1+\epsilon_2}{2}\right)+\cos^p(\eps_1+\eps_2)\ge 1,$$
	where the last inequality follows from Lemma \ref{LL}.
\end{proof}

We are now ready to prove Theorem \ref{thm:main}.

\begin{proof}[Proof of Theorem \ref{thm:main}] 
Let $X(\theta_1,\ldots,\theta_N)=\{x_1,\ldots,x_N\}$ be a global minimizer for $\fp_{p,N,2}$. Given the invariance of $\fp_{p, N,2}$ under rotations, permutations and phasing as mentioned in the introduction, we may assume without any loss of generality that  all $\theta_j\in [0,\pi)$ and view their  differences as real numbers modulo $\pi$. We will first prove $\Fc_{p,N,2}=\fp_{p,N,2}(\Xperp{N})= (N-1)^2/2$ by induction on odd $N$. We will address the uniqueness of the optimal configuration in the end.

 To set up the basis of the induction, we note that the result is trivial for $N=1$ and known for $N=3$, see \cite{EhlOko2012}. Thus we assume from now on that $N\ge 5$, and proceed to the induction step.

There are two cases to consider.

\noindent {\bf Case I:} $|\theta_i-\theta_j \mod\frac{\pi}{2}|\ge \frac{\pi}{6}$ for any $i\neq j$.

%All angles $\theta_i-\theta_j \mod \pi$ are away from $\frac{\pi}{2}$ by at least $\frac{\pi}{6}$. 

This is the easy case. Recall $W,V$ as defined in \eqref{equ:W} and \eqref{equ:V}. We know then that  $(\theta_i-\theta_j)\mod \pi\le \theta_p$ or $\ge\pi-\theta_p$ and hence  $W_p(\theta_i-\theta_j)\ge V(\theta_i-\theta_j)$ for all $i,j$. By Proposition \ref{prop:54} the potential function $F_V$ has global minimal value $(N-1)^2/2$, so this value is a lower bound on $\fp_{p,N,2}$. On the other hand it is achieved by $X_N^\perp$. We conclude that $F_V(\Theta)=F_{W_p}(\Theta)=(N-1)^2/2$. But necessarily $V(\theta_i-\theta_j)=W_p(\theta_i-\theta_j)$ for all $i,j$. This implies that $\theta_i-\theta_j\mod \pi =0,\theta_p,\pi-\theta_p$ for all $i,j$. After translation we assume w.l.o.g that $\Theta=(0,\ldots,0,\theta_p,\ldots,\theta_p, \pi-\theta_p,\ldots,\pi-\theta_p)$. For $p<\frac{\log 3}{\log 2}$ we have $\theta_p>\frac{\pi}{3}$, which contradicts our assumption that $|\theta_i-\theta_j \mod\frac{\pi}{2}|\ge \frac{\pi}{6}$, Unless $\Theta=(0,\ldots,0)$. But this is clearly not a minimizer. It remains for Case I to consider the scenario $p=\frac{\log 3}{\log 2}$ and $\theta_p=\frac{\pi}{3}$. Then $\Theta$ has $a$ 0's, $b$ angles $\frac{\pi}{3}$ and $c$ angles $\frac{2\pi}{3}$, $a+b+c=N$. We compute
$F_{W_p}(\Theta)=F_V(\Theta)=a(a-1)+b(b-1)+c(c-1)+2(ab+ac+bc)/3$. As a function of the real variables $a,b,c$, it is minimized when $a=b=c=N/3$, so $F_V(\Theta)\ge 5N^2/9-N>(N-1)^2/2$. Hence this scenario as well cannot occur.\\

\noindent{\bf Case II:} $|\theta_i-\theta_j \mod\frac{\pi}{2}|< \frac{\pi}{6}$ for some $i\neq j$.

With relabeling, rotation, and reflection, we assume $\theta_1=0$ and the angle $|\theta_2-\theta_1|$ is the closest to $\frac{\pi}{2}$. Let $\theta_2-\theta_1=\rho+\pit$ for some $0<\rho<\frac{\pi}{6}$. In order to make sure  the angle $|\theta_2-\theta_1|$ is the closest to $\frac{\pi}{2}$,  we must have $\{\theta_3, \cdots, \theta_N\}\subseteq [2\rho, \pit-\rho]\cup[\theta_2,\pi]$. Let the range $T_1:=[\theta_2,\pi]$ be of Type I angle, and the range $T_2:=[2\rho, \pit-\rho]$ be of Type II angle. See Figure \ref{fig:angles}.

 Notice that if we had $\rho>\frac{\pi}{6}$, there would be no room for Type II angles.

\begin{figure}[hbt]
\caption{Type I and Type II angles in Case II. The angles $\theta$ and $\alpha$ are used in the proof of Lemma \ref{caseIIa}}\label{fig:angles}
\begin{tikzpicture}
%rho=20 degree
\draw[pattern=dots,pattern color=red,draw=red!0] (0,0)--({3.7*cos(40)},{3.7*sin(40)})--({3.7*cos(40)},{3.7*sin(40)}) arc (40:70:3.7)--({3.7*cos(70)},{3.7*sin(70)})--(0,0);
\draw[pattern=dots,pattern color=blue,draw=red!0] (0,0)--({3.7*cos(110)},{3.7*sin(110)})--({3.7*cos(110)},{3.7*sin(110)}) arc (110:180:3.7)--({3.7*cos(180)},{3.7*sin(180)})--(0,0);
\draw[](-4,0)--(4,0);
\draw[](0,0)--(0,4);
\draw[ultra thick, -latex](0,0)--(3.7,0) node[below]{$x_1$};
\draw[ultra thick, -latex](0,0)--({3.7*cos(110)},{3.7*sin(110)}) node[above]{$x_2$};
\draw[thick, -latex,blue](0,0)--({3.7*cos(160)},{3.7*sin(160)}) node[left]{$x_i$};
\draw[ thick, -latex,red](0,0)--({3.7*cos(50)},{3.7*sin(50)}) node[right]{$x_j$};
\draw[red] ({1*cos(0)},{1*sin(0)}) arc (0:50:1);
\draw[red] (1.1,0.5) node{$\alpha$};
\draw[blue] ({1*cos(110)},{1*sin(110)}) arc (110:160:1);
\draw[blue] (-0.71,0.96) node{$\theta$};
\draw[blue](-2,1.65) node{Type I};
\draw[red](1.7,2.3) node{Type II};
%\draw[dashed](0,0)--({4*cos(70)},{4*sin(70)});
%\draw[dashed](0,0)--({4*cos(40)},{4*sin(40)});
\draw (1.5,0) arc (0:40:1.5);
\draw (1.7,0.6) node{$2\rho$};
\draw ({1.5*cos(70)},{1.5*sin(70)}) arc (70:90:1.5);
\draw (0.32,1.65) node{$\rho$};
\draw ({1.5*cos(90)},{1.5*sin(90)}) arc (90:110:1.5);
\draw (-0.30,1.65) node{$\rho$};

\end{tikzpicture}
\end{figure}
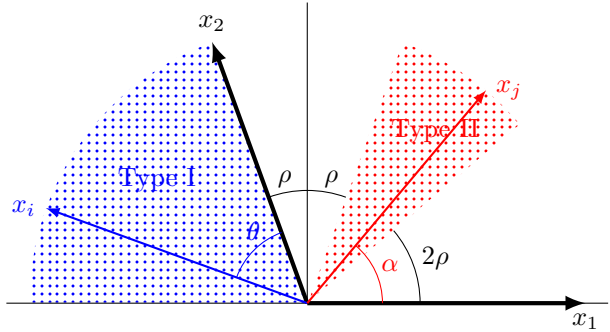

We will prove the existence of ${i_0},{j_0}$  such that 
\begin{equation}\label{diff}
    \sum_{i\neq i_0, j_0}^N \left(|\cos(\theta_i-\theta_{i_0})|^p+|\cos(\theta_i-\theta_{j_0})|^p\right) \ge (N-2).
\end{equation}

By the induction hypothesis, we  have $\fp_{p,N-2,2}(X\backslash\{x_{i_0},x_{j_0}\})\ge (N-3)^2/2$. If equation \eqref{diff} holds, then we will have
\begin{align*}
    \fp_{p,N,2}(X)&\ge 2(N-2)+\fp_{p,N-2,2}(X\backslash\{x_{i_0},x_{j_0}\})+2|\cos(\theta_{i_0}-\theta_{j_0})|^p\\
    & \ge 2(N-2)+\frac{(N-3)^2}{2}=\frac{(N-1)^2}{2}.
\end{align*}

It remains to prove \eqref{diff}.
In \cite{GP20} it was argued that for $p=1.3$, 
$$|\cos(\theta-\theta_{i_0})|^{1.3}+|\cos(\theta-\theta_{j_0})|^{1.3}\ge 1,\quad\text{for any }\theta\in T_1\cup T_2,$$  which was sufficient to make the induction step. For larger $p$ this is no longer true for $\theta\in T_2$, but the average over $N-2$ angles is still bigger than 1, i.e. \eqref{diff} still holds. We divide the discussion into two subcases.\\

%In order to prove \eqref{diff}, there are two cases to consider.

\noindent{\bf Case IIa:} Among $\{\theta_3, \cdots, \theta_N\}$, the number of Type I angles is at least as large as the number of Type II angles.  In this case, we will prove \eqref{diff} holds with $i_0=1, j_0=2$.

Define
$$\Delta_p(\theta):=|\cos(\theta_i-\theta_2)|^{p}+|\cos(\theta_i-\theta_1)|^{p}.$$

We will prove 

\begin{lemma}\label{caseIIa}
	\begin{itemize}
		\item[]
		\item[(a)] For all $p\le 2$ and $\theta\in T_1$, $\Delta_p(\theta)\ge 1$.
		\item[(b)] For all $p\le \frac{\log(3)}{\log(2)}$, and all $\theta_i\in T_1, \theta_j\in T_2$, $$\Delta_p(\theta_i)+\Delta_p(\theta_j)\ge 2.$$ 
	\end{itemize}
\end{lemma} 
By using this lemma we can pair each angle of Type II with an angle of Type I, and by part (b) this pair contributes to the potential function at least 2. The remaining Type I vectors contribute by part (a) at least 1. Thus in \emph{average}, $\Delta_p(\theta_i)$ is at least $1$, which is what we need to conclude equation \eqref{diff} and the induction step for Case IIa. To prove the uniqueness assertion notice that we must have equality in \eqref{diff}, and hence equality in parts (a) and (b) of the lemma. The uniqueness for $N-2$ and the equality statements in the lemma show that the only minimizer is $X_N^\perp$.  

We now turn to the proof of this lemma.

\begin{proof}[Proof of Lemma \ref{caseIIa}]

	Part (a) is a direct consequence of Lemma \ref{KK}. We proceed to prove (b). Let $\alpha$ be the angle $\theta_j-\theta_1=\theta_j$ and let $\theta=\theta_i-\theta_2$. We have $2\rho \le \alpha\le \pit-\rho$, $0\le \theta\le \pit-\rho$. See Figure \ref{fig:angles}. It can be computed that
	$$\Delta_p(\theta_i)+\Delta_p(\theta_j)=\cos^p\theta+\cos^p\left( \pit-\rho-\theta \right)+\cos^p\alpha+\cos^p\left( \pit+\rho-\alpha\right)
	$$
	Minimizing over $\theta\in \left[0,\pit-\rho\right]$, Lemma \ref{KK} implies that
	$$\Delta_p(\theta_i)+\Delta_p(\theta_j)\ge 1+\sin^p\rho+\cos^p\alpha+\sin^p(\alpha-\rho).$$
	Now we fix $\alpha$ and minimize over $\rho$. This function is symmetric to the change of variables $\alpha\to \pit+\rho-\alpha$. By making this change, if necessary, we may assume that $\alpha\le \frac{\pi}{3}$. Notice that in both cases $0\le \rho \le \frac{\alpha}{2}$. By Lemma \ref{RR}, we now have
	
	$$\Delta_p(\theta_i)+\Delta_p(\theta_j)\ge 1+\cos^p\alpha+\min\left(\sin^p\alpha,2\sin^p\left(\frac{\alpha}{2}\right)\right).$$
	
	Now, we have 
	$$1+\cos^p\alpha+\sin^p\alpha\ge 1+\cos^2\alpha+\sin^2\alpha= 2,$$
	and by Lemma \ref{LL}
	$$1+\cos^p\alpha+2\sin^p\left(\frac{\alpha}{2}\right)\ge 2.$$
	This completes the proof of the inequality (b).\\
\end{proof}

	\noindent{\bf Case IIb:}
	Among $\{\theta_3, \cdots, \theta_N\}$, the number of Type II angles is larger than the number of Type I angles. 
	
	In this case we will take one of the angles of Type II, say $\theta_3$, which is the closest to $\theta_1=0$. We have $\theta_3=2\rho+\alpha$ for some $\alpha\ge 0$. Note that $\alpha\le \pit-3\rho$. In this case, we will prove \eqref{diff} holds with $\{i_0,j_0\}=\{2,3\}$.
	
	Every other angle of Type II, $\theta_i$, $i\neq 1,2,3$ is of the form $\theta_i=2\rho+\alpha+\beta$ for $0\le \beta\le \pit-3\rho-\alpha$. We will consider $\theta_1=0$ as an angle of Type I. %We will carry the induction step by considering all angles except for $\theta_2$ and $\theta_3$, but be aware that the terminology of Type I and Type II is still with respect to the old setting. 
	As before, it will suffice to show that the average value of
	$$E_p(\theta)=|\cos(\theta-\theta_2)|^p+|\cos(\theta-\theta_3)|^p$$
	over the multiset of $\theta=\theta_i$, $i\neq 2,3$ is at least 1. With this in mind we have
	
	\begin{lemma}\label{lem:type2}
		Given $1\le p\le 2$, every angle $\theta_i$, $i\neq 2,3$ of Type II satisfies $E_p(\theta)\ge 1$.
	\end{lemma}

	\begin{proof}[Proof of Lemma \ref{lem:type2}]
		This is just an application of Lemma \ref{KK} with $\theta=\beta$ and $\nu=\rho+\alpha\le \pit$. 
	\end{proof}

	In view of this, and since the number of Type II angles is at least as large as the number of Type I angles (including $\theta_1$), it suffices to show that for every $\theta_i$ of Type II and every $\theta_j$ of Type I, we will have
	$$E_p(\theta_i)+E_p(\theta_j)\ge 2.$$ The remainder of this proof is devoted to proving this assertion. Let 
	$$ \theta_i=2\rho+\alpha+\beta, \ \ \text{and } \ \ \theta_j=\pit+\rho+\theta.$$
	We have $0\le\beta\le \pit-3\rho-\alpha$ and $0\le \theta\le \pit-\rho$. Then
	$$E_p(\theta_i)+E_p(\theta_j) =\cos^p\beta+\cos^p\left(\pit-\rho-\alpha-\beta\right)+\cos^p\theta+\left| \cos\left( \pit+\alpha+\rho-\theta\right)\right|^p.$$ We now make use of Lemma \ref{KK} to minimize over $\beta$ and conclude that 
	$$E_p(\theta_i)+E_p(\theta_j)\ge 1+\sin^p(\alpha+\rho)+\cos^p\theta+\left| \cos\left( \pit+\alpha+\rho-\theta\right)\right|^p.$$ Write $\alpha+\rho=\gamma$. Then we rewrite
	\begin{equation}\label{eqIIb}
		E_p(\theta_i)+E_p(\theta_j)\ge 1+\sin^p(\gamma)+\cos^p\theta+\left| \cos\left( \pit+\gamma-\theta\right)\right|^p,
	\end{equation} and we shall consider two cases: Case IIb.1 with $\gamma\le \theta$ and Case IIb.2 with $\gamma>\theta$.\\
	
	In Case IIb.1, noting that $\theta-\gamma\le \pit$ we can drop the absolute value sign in Equation \eqref{eqIIb}. Define $\delta=\theta-\gamma$. Then \eqref{eqIIb} can be rewritten as
	\begin{equation}
		E_p(\theta_i)+E_p(\theta_j)\ge 1+\sin^p\gamma+\sin^p\delta+\sin^p\left(\pit-\gamma-\delta\right),
	\end{equation} 
	but note that the three angles $\gamma,\delta$ and $\pit-\gamma-\delta$ are acute and sum up to $\pit$. Now Lemma \ref{lem:eps} tells us that $E_p(\theta_i)+E_p(\theta_j)\ge 2$. So Case IIb.1 is settled.\\
	
	We turn to Case IIb.2, where $\theta<\gamma$. Then \eqref{eqIIb} is written as 
	$$ E_p(\theta_i)+E_p(\theta_j)\ge 1+\sin^p\gamma+\cos^p\theta+\sin^p(\gamma-\theta).$$ But by Lemma \ref{KK}, using that $\theta$ and $\pit-\gamma$ are acute and their sum $\theta+\pit-\gamma<\pit$, we learn that 
	$\sin^p\gamma+\cos^p\theta=\cos^p\left(\pit-\gamma\right)+\cos^p\theta \ge 1$, which proves $E_p(\theta_i)+E_p(\theta_j)\ge 2$. This completes the proof of Case IIb.2, and thus completes the proof of the induction step. We have shown that the absolute minimum value of $ \fp_{p,N,2}(X)$ is $(N-1)^2/2$.\\
	
	It remains to address the question of uniqueness. This will also be proved by induction on $N$. For $N=1$ the assertion is trivial. For the induction step, suppose that $(\theta_1,\ldots,\theta_N)$ is a global minimizer. Then we choose $\theta_1,\theta_2$ as above, and we know that inequality \eqref{diff} is satisfied. But a posteriori, it must be an equality, $(\theta_3,\ldots,\theta_N)$ must be a minimizer for $\fp_{p,N-2,2}$, and $\theta_1,\theta_2$ must be orthogonal. If $N=3$ then by Lemma \ref{KK}  $\theta_3=\theta_1,\theta_2$ and we are done. If $N>3$ we know by the induction hypothesis that $(\exp(i\theta_3),\ldots,\exp(i\theta_N))=X_{N-2}^\perp$. We may remove any other two perpendicular vectors, and conclude similarly that the remaining configuration is an $X_{N-2}^\perp$. This implies that any two vectors in the configuration are either equal or perpendicular, and thus it must be an $X_N^\perp$. The proof is complete. 
\end{proof}

\begin{remark}
Theorem \ref{thm:main} is not sharp as suggested by the numerical results in Section \ref{sec3}. The result is restricted by our proof techniques. For $N=5$, we believe the result holds for $p$ up to $\approx1.77$. 
\end{remark}

\section{Asymptotic behaviour of the minimizers and Numerical results}\label{sec3}
This section is divided into two parts. First, we look at the asymptotic behaviour of the minimizers of~\eqref{equ:min}, and then present some numerical results of the solution to our problem when $N\in \{5, 6, 7\}$.

\subsection{Asymptotic behaviour of the minimizers}
As mentioned above,  for even integers $N=2k$ and $p\in (0, 2)$, the solution to~\eqref{equ:min} is a configuration that consists of $k$ copies of any orthonormal basis of $\R^2.$ However, as shown in the last section, when $N=2k+1$ the minimizers of the $p^{th}$ frame potentials for $p\in (0,2)$ are more difficult to classify. Nonetheless, in this section we prove that as $N\to \infty,$ these minimizers approach copies of an orthonormal basis in $\R^2$. To establish this result, we first need to introduce a continuous analog of the frame potential that which is interesting in its own right.

Given a probabilistic measure $\mu$ on the unit sphere in $\R^d$, the \emph{probabilistic $p$ frame potential} is defined as
\begin{equation}
\pfp_{p,d}(\mu):=\int_{\Sb^{d-1}}\int_{\Sb^{d-1}}|\langle x,y\rangle|^pd\mu(x)d\mu(y).
\end{equation}

Let $\mathcal{M}(\Sb^{d-1})$ be the collection of all probabilistic measures on the sphere. The relationship between the discrete problem \eqref{equ:min} and the continuous one
\begin{equation}\label{equ:pfp}
\Pc_{p,d}:=\min_{\mu\in\mathcal{M}(\Sb^{d-1})}\pfp_{p,d}(\mu)
\end{equation}
was explored in \cite{CGGKO}. We also encourage the readers to see a general version with a general potential kernel in \cite[Theorem 4.2.2]{MEbook}.

\begin{proposition}[{\cite[Proposition 2.6]{CGGKO}}]\label{prop:asymp}
Given $d\geq 2$ and  $p\in(0,\infty)$. For every $N\geq2$, let $X_N$ be an $N$-point optimal configuration of  \eqref{equ:min}, then every weak star cluster point $\nu^*$ of the normalized counting measure $\nu_{X_N}=\frac{1}{N}\sum_{x\in X_N}\delta_x$ solves \eqref{equ:pfp}, that is $\pfp_{p,d}(\nu^*)=\Pc_{p,d}$. 
\end{proposition}

\begin{theorem}\label{thm:tight}
Let $0<p<2$. If   $X_N=\left \{x_1,\ldots,x_N\right\}\in S(N,2)$ is  an optimal configuration for \eqref{equ:min} for $N\ge 2$, then (treated as a matrix)
\begin{equation}\label{equ:asytight}
\lim_{N\rightarrow\infty}\frac{1}{N}X_NX_N^T=\frac{1}{2}I_2.\end{equation} In particular, we have $$\|X_NX_N^T -\tfrac{N}{2}  I_2\|=o(N).$$
%{\color{red} So $\simeq$ means up to an error term of size $o(N)$.}
\end{theorem}

\begin{proof}
When $d=2$ and $0<p<2$, it is proved in \cite[Theorem 4.9]{EhlOko2012} that the unique minimal measure of \eqref{equ:pfp} is the normalized counting measure on $\{e_1,e_2\}$, denoted by $\sigma_e$. Since every cluster point of $\{\nu_{X_N}\}$ weak* converges to $\sigma_e$, $\{\nu_{X_N}\}$ itself weak* converges to $\sigma_e$. Thus  for every continuous function $f$ defined on $\Sb^{1}$, 
\begin{equation}\label{equ:ws}\lim_{N\rightarrow\infty}\int_{\Sb^{1}} f(x,y)d\nu_{X_N}=\int_{\Sb^{1}} f(x,y)d\sigma_e=\frac{1}{2}(f(1,0)+f(0,1)).
\end{equation}
 
Let $f(x,y)=x^2$. We rewrite the points in $X_N$ as $x_i=(x_{i1},x_{i2})$, then  \eqref{equ:ws} becomes 
\begin{equation}\label{equ:xx}
\lim_{N\rightarrow\infty}\frac{1}{N}\sum_{i=1}^Nx_{i1}^2=\frac{1}{2}.
\end{equation}

Let $f(x,y)=y^2$. Then  \eqref{equ:ws} becomes 
\begin{equation}\label{equ:yy}
\lim_{N\rightarrow\infty}\frac{1}{N}\sum_{i=1}^Nx_{i2}^2=\frac{1}{2}.
\end{equation}

Let $f(x,y)=xy$. Then  \eqref{equ:ws} becomes 
\begin{equation}\label{equ:xy}
\lim_{N\rightarrow\infty}\frac{1}{N}\sum_{i=1}^Nx_{i1}x_{i2}=0.
\end{equation}

The combination of \eqref{equ:xx}, \eqref{equ:yy}, and \eqref{equ:xy} implies \eqref{equ:asytight}.
\end{proof}

We have seen that $\Xperp{3}$ is the unique global minimizer of $\fp_{p,3,2}$ for $p\in(0, \frac{\log 3}{\log 2})$, and after $p=\frac{\log 3}{\log 2}$, the optimal configuration transitions to $X^{(h)}_3$. When $N$ is increased to 5, this phase transitioning happens at a larger value of $p\approx 1.77766$ (See Conjecture~\ref{conjecn5}). When $N=7$, our numerical experiments suggest that this transitioning happens at $p\approx1.84$.

Let 
$$p(N):=\sup\{p_0: \fp_{p,N,2}(X_N^\perp)=\Fc_{p,N,2}\text{ for }0<p<p_0\}.$$
%Observe for large values of odd $N$, that $X_N^\perp$ is very close to being a tight frame. Then the value $\fp_{2,N,2}(X_N^\perp)$ is very close the the global minimum. \textcolor{red}{$p(3)=1.58, p(5)=1.77766, p(7)=1.84$. Based on this observation and the extensive numerical experiments we ran},  
We pose the following  conjecture: 

\begin{conjecture}\label{conj:1}
    $p(2k+1)$ is monotone increasing as $k\geq1$ increases, and
    $$\lim_{ k\to\infty} p(2k+1)=2.$$
\end{conjecture}

\subsection{Numerical Results}
In this section we present some numerical results about the solutions to the minimization of the frame potential \eqref{eq:pfrpot}. In particular, we used the MATLAB minimization function \verb"fminsearch"  with random initial configuration. In order to avoid local minima we ran the minimization 3000 times for every value of $p$ and we chose the configuration with minimal potential (we ignored the cases where the function returned an error with the indication of not being able to find a minimum).  For the code of these experiments, we refer the readers to~\cite{radel}. One of the outcomes of these numerical results is  Conjecture~\ref{conjecn5} about the minimizers of $\fp_{p, 5, 2}$ for $p\in (0, \infty)$. We performed similar numerical experiments for other odd values of $N$, but the behavior of the minimizers seems to get more intricate. In particular, as $p$ increases, the number  phase transitions seems to be increasing with $N$. As such we first focus on the case  $N=5$, for which we consider the following special configurations. In Theorem~\ref{thm:main} we have established that $X_5^\perp$ is the minimizer for $\Fc_{p, 5, 2}$ for all $p\leq \frac{\log 3}{\log 2}\approx 1.58$.

 Define
\begin{itemize}
    \item $Y(\alpha)=X\{0,0,\frac{\pi}{2}-\alpha,\frac{\pi}{2},\frac{\pi}{2}+\alpha\}$
    \item $Z(\alpha)=X\{-\alpha,-\alpha,0,\alpha,\alpha  \}$.
\end{itemize}

\begin{conjecture}\label{conjecn5}
    The absolute minimizer for $\fp_{p,5,2}(X)$ is given by
    $$\begin{cases}
    X_5^\perp & 0\le p\le p_1,\\
    Y(\alpha(p)) & p_1\le p\le p_2 \\
    Z(\alpha(p)) & p_2 \le p \le 2\\
    X_5^{(h)} & p\geq2
    \end{cases}.
    $$
    where $p_1=1.77766251887019$ and $p_2=1.78329970946521$ are given to a precision of $10^{-14}$.
    Moreover, this minimizer is unique up to rotation and antipodal reflections, for any $p$ not in endpoints of the intervals.
    
\end{conjecture}

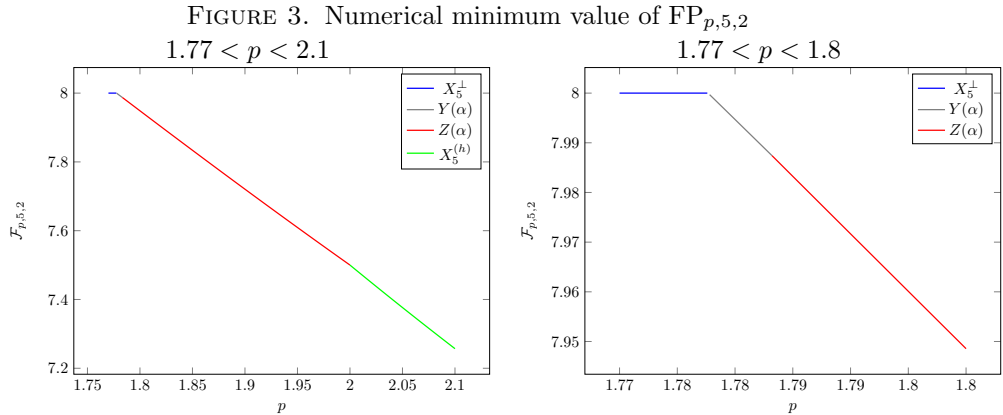
\begin{figure}[hbt]
 \caption{Numerical minimum value of $\fp_{p,5,2}$}
 \label{fig:minval}
\begin{tabular}{cc}
$1.77<p<2.1$ & $1.77<p<1.8$\\
 \begin{tikzpicture}[scale=0.6]
        \begin{axis}[
            every axis plot/.append style={thick},
            xlabel=$p$,
            ylabel=$\mathcal F_{p,5,2}$,
            width=0.5*\paperwidth,
            height=0.3*\paperheight]
            \addplot + [mark=none, color=blue,restrict x to domain=1.75:1.77766] table [x=$p$, y expr=\thisrow{$f_{min}$}-5] {data.dat};
            \addlegendentry{$X_5^\perp$}
            \addplot + [mark=none, color=gray,restrict x to domain=1.77766:1.7832] table [x=$p$, y expr=\thisrow{$f_{min}$}-5] {data.dat};
            \addlegendentry{$Y(\alpha)$}
            \addplot + [mark=none, color=red,restrict x to domain=1.7832:2.0] table [x=$p$, y expr=\thisrow{$f_{min}$}-5] {data.dat};
            \addlegendentry{$Z(\alpha)$}
            \addplot + [mark=none, color=green,restrict x to domain=2.0:2.1] table [x=$p$, y expr=\thisrow{$f_{min}$}-5] {data.dat};
            \addlegendentry{$X_5^{(h)}$}
%            \addplot [black, mark options={scale=0.3},mark = *, nodes near coords=$ $,every node near coord/.style={anchor=180}] coordinates {( 1.777, 12.2)};
%            \addplot [black, thin, mark options={scale=0.3}, mark = *, nodes near coords=$ $,every node near coord/.style={anchor=180}] coordinates {( 1.7832, 12.2)};
%            \addplot [black, mark = *, nodes near coords={$p=2$},every node near coord/.style={anchor=180}] coordinates {( 2.0, 12.5)};
             \end{axis}; 
    \end{tikzpicture}&
     \begin{tikzpicture}[scale=0.6]
        \begin{axis}[
            every axis plot/.append style={thick},
            xlabel=$p$,
            ylabel=$\mathcal F_{p,5,2}$,
            width=0.5*\paperwidth,
            height=0.3*\paperheight]
            \addplot + [mark=none, color=blue,restrict x to domain=1.75:1.77766] table [x=$p$, y expr=\thisrow{$f_{min}$}-5] {data.dat};
            \addlegendentry{$X_5^\perp$}
            \addplot + [mark=none, color=gray,restrict x to domain=1.77766:1.7832] table [x=$p$,y expr=\thisrow{$f_{min}$}-5] {data.dat};
            \addlegendentry{$Y(\alpha)$}
            \addplot + [mark=none, color=red,restrict x to domain=1.7832:1.8] table [x=$p$, y expr=\thisrow{$f_{min}$}-5] {data.dat};
            \addlegendentry{$Z(\alpha)$}
            \end{axis}
    \end{tikzpicture}

\end{tabular}

 \end{figure}

Conjecture~\ref{conjecn5} is illustrated in Figure~\ref{fig:minval} with the plot of   the minimal potential as a function of $p$, and 
%The minimal potential as a function of $p$ is plotted in Figure \ref{fig:minval}. 
the type of the minimal configuration.
The value of $\alpha$  in $Y(\alpha)$ and $Z(\alpha)$ as a function of $p$ is shown in Figure \ref{fig:typeYZ}.
It appears that the transition of the configuration at points $p=p_1$ and $p=p_2$ and $p=2$ is non-continuous. Note that using Table \ref{tab:R2known}, we only have $X_5^{(h)}$ is a unique minimizer for $p>6$. 

In addition, there is a discontinuity in the derivative $\frac{d \mathcal F_{p,5,2}}{dp}$ as shown in Figure \ref{fig:der}(a). Note that the tiny zig-zag at the bottom left corner of Figure \ref{fig:der}(a) corresponds to phase transitioning at $p_1, p_2$, as described in Conjecture \ref{conjecn5}. The derivative plot is effective at locating the phase transitions in general.

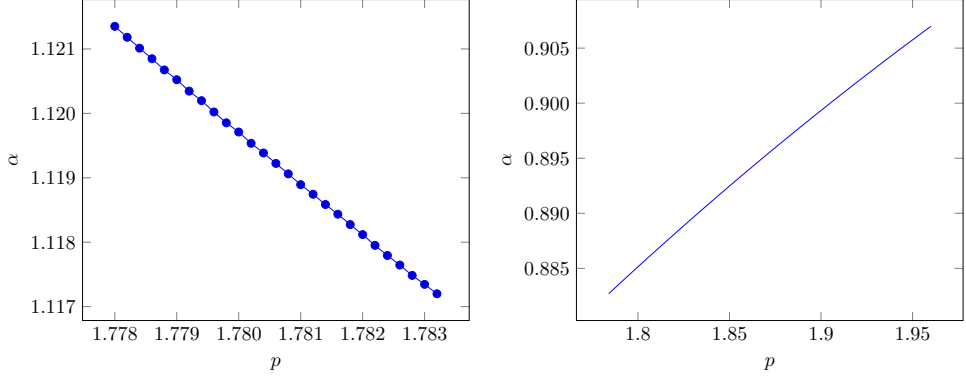
\begin{figure}[hbt]
\caption{The value of angles for Type $Y$ and Type $Z$}
\label{fig:typeYZ}
\begin{tabular}{cc}
Type $Y$&Type $Z$\\
 \begin{tikzpicture}[scale=0.75]
        \begin{axis}[
            xlabel=$p$,
            ylabel=$\alpha$,
            y tick label style={
        /pgf/number format/.cd,
            fixed,
            fixed zerofill,
            precision=3,
        /tikz/.cd
    },
    x tick label style={
        /pgf/number format/.cd,
            fixed,
            fixed zerofill,
            precision=3,
        /tikz/.cd
    },]
            \addplot + [ color=blue, restrict x to domain=1.778:1.7832] table [x=$p$, y=$c$]{data.dat};
        \end{axis}
    \end{tikzpicture}&
    \begin{tikzpicture}[scale=0.75]
        \begin{axis}[
            xlabel=$p$,
            ylabel=$\alpha$,
            y tick label style={
        /pgf/number format/.cd,
            fixed,
            fixed zerofill,
            precision=3,
        /tikz/.cd
    },
   ]
            \addplot + [mark=none, color=blue, restrict x to domain=1.784:1.96] table [x=$p$, y=$d$]{data.dat};
        \end{axis}
    \end{tikzpicture}

\end{tabular}
\end{figure}

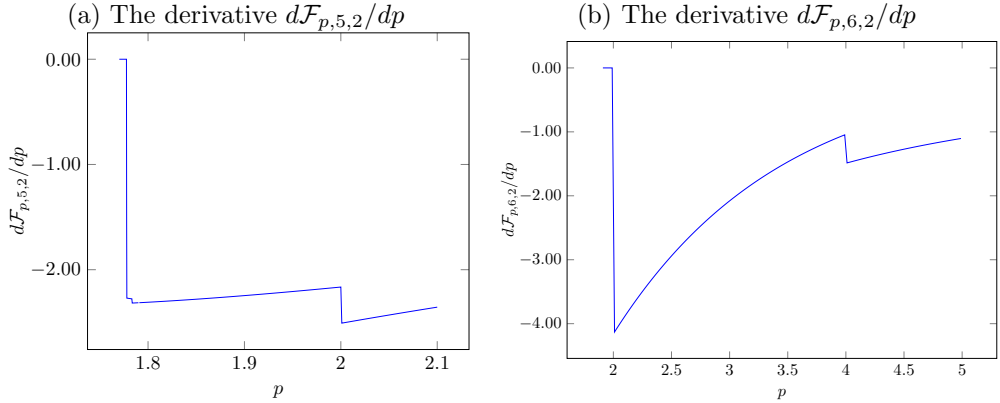
\begin{figure}[hbt]
\caption{Derivatives of the minimal frame potential}
\label{fig:der}
\begin{tabular}{cc}
(a) The derivative $d\mathcal F_{p,5,2}/dp$&(b) The derivative $d\mathcal F_{p,6,2}/dp$\\
 \begin{tikzpicture}[scale=0.74]
        \begin{axis}[
            xlabel=$p$,
            ylabel=$d\mathcal F_{p,5,2}/dp$,
            y tick label style={
        /pgf/number format/.cd,
            fixed,
            fixed zerofill,
            precision=2,
        /tikz/.cd
    },
   ]
            \addplot + [mark=none, color=blue,restrict x to domain=1.77:2.1] table [x=$p$, y=der]{data.dat};
        \end{axis}
    \end{tikzpicture}&
    \begin{tikzpicture}[scale=0.62]
        \begin{axis}[
            xlabel=$p$,
            ylabel=$d\mathcal F_{p,6,2}/dp$,
            width=0.5*\paperwidth,
            height=0.3*\paperheight,
            y tick label style={
        /pgf/number format/.cd,
            fixed,
            fixed zerofill,
            precision=2,
        /tikz/.cd
    },
   ]
            \addplot + [mark=none, color=blue] table [x=$p$, y=$f_{min}$]{pvec6b.dat};
        \end{axis}
    \end{tikzpicture}
    \end{tabular}
\end{figure}

The derivative was numerically estimated using symmetrical difference as $f'(x)\approx(f(x+d)-f(x-d))/2d$. For comparison we have added the plots of the derivative of the potential for $N=6$ points  (Figure \ref{fig:der}(b)) and for $N=7$ points (Figures \ref{fig:7pts}). Notice the few phase transitions there. 

For $N=6$  the first phase transition is at $p=2$ (as expected) and a second at $p=4$. Up to $p=2$ the optimal configuration is $X_6^\perp$. For $2\leq p \leq 4$ the optimal configuration is $E=X\{0,0,\frac{\pi}{3},\frac{\pi}{3},\frac{2\pi}{3},\frac{2\pi}{3}\}$. For $p\geq 4 $ the optimal configuration is $X^{(h)}_6 $ .

For $N=7$ It seems that there are more phase transitions, at $p\approx 1.84,  p=2, p\approx 3.5$ and $p=4$. Up to $p\approx 1.84$ the optimal configuration is $X^{\perp}_7$. For $p>4$ it is $X^{(h)}_7$. We have not characterized the  minimal configuration for the phases in between.

    \begin{figure}[hbt] 
    \caption{The derivative $d\mathcal F_{p,7,2}/dp$}
     \label{fig:7pts}
        \begin{tabular}{cc}
        
                \begin{tikzpicture}[scale=0.57]
            \begin{axis}[
                xlabel=$p$,
                ylabel=$d \mathcal F_{p,7,2}/dp$,
                width=0.5*\paperwidth,
                height=0.3*\paperheight,
                y tick label style={
            /pgf/number format/.cd,
                fixed,
                fixed zerofill,
                precision=2,
            /tikz/.cd
        },
       ]
                \addplot + [mark=none, color=blue] table [x=$p$, y=$f_{min}$]{pvec7a.dat};
            \end{axis}
        \end{tikzpicture}&           
       \begin{tikzpicture}[scale=0.57]
            \begin{axis}[
                xlabel=$p$,
                ylabel=$d\mathcal F_{p,7,2}/dp$,
                width=0.5*\paperwidth,
                height=0.3*\paperheight,
                y tick label style={
            /pgf/number format/.cd,
                fixed,
                fixed zerofill,
                precision=2,
            /tikz/.cd
        },
       ]
                \addplot + [mark=none, color=blue] table [x=$p$, y=$f_{min}$]{pvec7b.dat};
            \end{axis}
        \end{tikzpicture}
        \end{tabular}
    \end{figure}
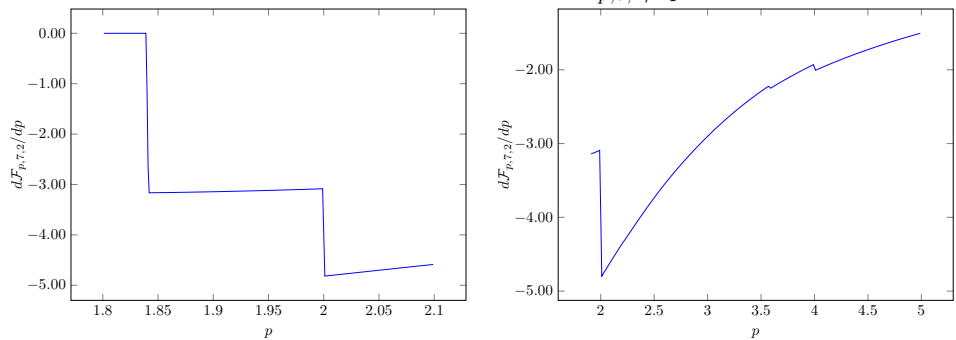

\bibliographystyle{amsplain}
\bibliography{pfpotential} 
\end{document}